\documentclass[11pt,reqno]{amsart}
\usepackage{graphicx}
\usepackage{float}
\usepackage{amsfonts}
\usepackage[caption = false]{subfig}
\usepackage{amsmath}
\usepackage{graphics}
\usepackage{float}
\usepackage{enumerate}
\usepackage{mathtools}
\usepackage{amsthm}
\usepackage[english]{babel}
\usepackage{amsfonts,amssymb}% insert here the call for the packages your document requires
\usepackage{mathrsfs}
\usepackage{tabularx}
\usepackage[utf8x]{inputenc}\usepackage[english]{babel}
\usepackage{soul}
\usepackage[all]{xy,xypic}
\usepackage{cite}
\usepackage{relsize}
\numberwithin{equation}{section}
\newtheorem{theorem}{Theorem}[section]
\newtheorem{corollary}[theorem]{Corollary}
\newtheorem{lemma}[theorem]{Lemma}

\theoremstyle{definition}
\newtheorem{definition}[theorem]{Definition}
\theoremstyle{remark}
\newtheorem{remark}[theorem]{Remark}

\numberwithin{equation}{section}
\DeclareMathOperator{\RE}{Re}
\DeclareMathOperator{\IM}{Im}

\begin{document}
	%\fontsize{14pt}{16pt}\selectfont
	\title[On sharp third Hankel determinant for certain starlike functions ]{On sharp third Hankel determinant for certain starlike functions}
	\author[Neha Verma]{Neha Verma}
	\address{Department of Applied Mathematics, Delhi Technological University, Delhi--110042, India}
	\email{nehaverma1480@gmail.com}
	\author[S. Sivaprasad Kumar]{S. Sivaprasad Kumar}
	\address{Department of Applied Mathematics, Delhi Technological University, Delhi--110042, India}
	\email{spkumar@dce.ac.in}

	
	\maketitle
\begin{abstract}
In this paper, we provide an estimation for the sharp bound of the third Hankel determinant of starlike functions of order $\alpha$, where $\alpha$ ranges in the interval $[0, 1/6]\cup \{1/2\}$ and thereby extending the result of Rath et al. (Complex Anal Oper Theory: No. 65, 16(5), 8 pp 2022).

\end{abstract}
\maketitle
	
\section{Introduction}
	\label{intro}
Consider the set $\mathcal{A}$, which comprises normalized analytic functions defined on the open unit disk $\mathbb{D}:=\{z\in \mathbb{C}:\vert z \vert<1\}$. These functions are expressed in the form:
\begin{equation}
	f(z) = z+\sum_{n=2}^{\infty}a_nz^n.\label{form}
	\end{equation}
Let $\mathcal{S}\subset\mathcal{A}$, where $\mathcal{S}$ represents the class of univalent functions, and $\mathcal{P}$ denotes the collection of analytic functions defined on $\mathbb{D}$ with a positive real part, expressed as $p(z)=1+\sum_{n=1}^{\infty}p_n z^n$.
 Let $h$ and $g$ are two analytic functions, then we say $h$ is subordinated to $g$, denoted as $h\prec g$, provided there exist a Schwarz function $w$, adhering to two crucial conditions: $w(0)=0$ and $\vert w(z)\vert \leq \vert z\vert $, such that $h(z)=g(w(z))$.

In the year 1936, Robertson \cite{robertson} introduced the class of starlike functions of order $\alpha$, characterized as follows:
\begin{definition}\cite{robertson}
For $0\leq \alpha<1$, we say that a function $f\in \mathcal{A}$ is starlike of order $\alpha$ if and only if
\begin{equation*}
    \RE\bigg(\dfrac{zf'(z)}{f(z)}\bigg)>\alpha,\quad z\in \mathbb{D}.
\end{equation*}
The class of all such functions is represented by $\mathcal{S}^{*}(\alpha)$.
\end{definition}
In 1992, Ma and Minda \cite{ma-minda} introduced a more general class of starlike functions through subordination, defined as follows:
\begin{equation*}
		\mathcal{S}^{*}(\varphi)=\bigg\{f\in \mathcal {A}:\dfrac{zf'(z)}{f(z)}\prec \varphi(z) \bigg\},\label{mindaclass}
\end{equation*}
where $\varphi$ is an analytic univalent function such that $\RE\varphi(z)>0$, $\varphi(\mathbb{D})$ is symmetric about the real axis and starlike with respect to $\varphi(0)=1$ with $\varphi'(0)>0$. Through this concept, we can re-define the class $\mathcal{S}^{*}(\alpha)$ as:
\begin{equation*}
\mathcal{S}^{*}(\alpha)=\bigg\{f\in \mathcal {A}:\dfrac{zf'(z)}{f(z)}\prec \frac{1+(1-2\alpha)z}{1-z}, \quad\alpha\in [0,1) \bigg\}.
\end{equation*}
Note that $\mathcal{S}^{*}(0)=\mathcal{S}^{*}$ and $\mathcal{S}^{*}(\varphi)\subset\mathcal{S}^{*}(\alpha)$ for some $\alpha$ depending upon the choice of $\varphi$.

The Bieberbach conjecture, as documented on \cite[Page no. 17]{goodman vol1}, has been a significant source of inspiration in the development of univalent function theory and in the formulation of coefficient problems. Building on this foundation, in 1966, Pommerenke \cite{pomi} introduced the concept of $qth$ Hankel determinants, denoted as $H_{q}(n)$, where $n$ and $q$ are both natural numbers, associated with analytic functions as in \eqref{form}, defined as follows:
\begin{equation}
		H_{q}(n) =\begin{vmatrix}
			a_n&a_{n+1}& \ldots &a_{n+q-1}\\
			a_{n+1}&a_{n+2}&\ldots &a_{n+q}\\
			\vdots& \vdots &\ddots &\vdots\\
			a_{n+q-1}&a_{n+q}&\ldots &a_{n+2q-2}\label{5hqn}
		\end{vmatrix}.
\end{equation}
By choosing specific values for both $n$ and $q$, we can examine particular cases of this concept. For instance, when we set $q=2$, we obtain the expression for the second-order Hankel determinant. Numerous studies have investigated and established sharp bounds for second-order Hankel determinants and other determinants within various subclasses of $\mathcal{S}$, see \cite{krishna bezilevic,secondhankelalpha,2 Janteng} for more details. Now, if we choose $q=3$ and $n=1$ in (\ref{5hqn}), assuming $a_1:=1$, we arrive at the expression for the Hankel determinant of order three, given by
\begin{equation}
H_{3}(1):=\begin{vmatrix}
$$1&a_2&a_3\\
a_2&a_3&a_4\\
a_3&a_4&a_5$$
\end{vmatrix}=a_3(a_2a_4-a_3^2)-a_4(a_4-a_2a_3)+a_5(a_3-a_2^2).\label{1h3}
%2 a_2a_3a_4-a_3^3-a_4^2-a_2^2a_5+a_3a_5.\label{1h3}
		%a_3(a_2a_4-a_3^2)-a_4(a_4-a_2a_3)+a_5(a_3-a_2^2)
\end{equation}

Determining the third-order Hankel determinant poses a greater challenge compared to the second-order, as evidenced in \cite{zap, kumar-ganganiaCardioid-2021}. We also list some of the sharp estimates for the third-order Hankel determinant concerning functions within the class $\mathcal{S}^{*}(\varphi)$, considering various selections of $\varphi(z)$ in Table \ref{6 table1}. However, the sharp estimate of $H_{3}(1)$ for $\mathcal{S}_{Ne}^{*}$ is yet to be estimated.

\begin{table}\label{6 table1}
\caption{\centering List of sharp third order Hankel determinants}
\centering
\begin{tabular}{|c| c |c|}
 \hline
 Class& Sharp bound & Reference \\ [1ex]
 \hline
  $\mathcal{S}^{*}:=\mathcal{S}^{*}(0)$  & 4/9&\cite{banga,4/9}    \\
 $\mathcal{S}^{*}(1/2)$&   1/9 &\cite{lecko 1/2 bound,rath}   \\
 $\mathcal{S}^{*}_{\varrho}:=\mathcal{S}^{*}(1+ze^z)$ & 1/9 &\cite{nehacardioid}\\%,\cite[Conjecture on Page no. 33]{kumar-ganganiaCardioid-2021}\\
  $\mathcal{SL}^{*}:=\mathcal{S}^{*}(\sqrt{1+z})$ & 1/36 &\cite{sharp}\\
  $\mathcal{S}^{*}_{e}:=\mathcal{S}^{*}(e^z)$ & 1/9&\cite{nehaexpo} \\
  $\mathcal{S}^{*}_{\rho}:=\mathcal{S}^{*}(1+sinh^{-1}(z))$ & 1/9&\cite{nehapetal} \\
  $\mathcal{S}_{Ne}^{*}:=\mathcal{S}^{*}(1+z-z^3/3)$ & ---&---\\ [1ex]
 \hline
\end{tabular}
\label{table1}
\end{table}
For the class $\mathcal{S}^{*}(\alpha)$, Krishna and Ramreddy \cite{secondhankelalpha} computed the bound of the second order Hankel determinant, $\vert a_2a_4-a_3^2\vert \leq (1-\alpha)^2$, $\alpha\in[0,1/2]$ while Xu and Fang  \cite{feketealpha} calculated the sharp bounds of the Fekete and Szeg\"{o} functional $\vert a_3-\lambda a_2^2\vert \leq (1-\alpha)\max\{1,\vert 3-2\alpha-4\lambda(1-\alpha)\vert \}$, $\lambda\in \mathbb{C}$ and $\alpha\in[0,1)$. We refer \cite{moreofalpha} for further information on Hankel determinants associated with the class $\mathcal{S}^{*}(\alpha)$.

The purpose of this study is to establish the sharp bound of third order Hankel determinant for functions belonging to the class, $\mathcal{S}^{*}(\alpha)$. At the end of this paper, we demonstrate the validation of our main result by considering the class $\mathcal{S}^{*}(\alpha)$ specifically for the case when $\alpha=0$, and we also present some relevant applications.

\subsection{Preliminary}
In this part of the section, we mention the initial coefficient bounds $a_i$ $(i=2,3,4,5)$ in terms of the Carath\'{e}odory coefficients and a lemma which will be used in our forthcoming results.
Let $f\in \mathcal{S}^{*}(\alpha)$, then a Schwarz function $w(z)$ exists such that
\begin{equation}\label{3 formulaai}
	\dfrac{zf'(z)}{f(z)}=\frac{1+(1-2\alpha)w(z)}{1-w(z)}.
\end{equation}
Let $p(z)=1+\sum_{n=2}^{\infty}p_nz^n \in \mathcal{P}$ and $w(z)=(p(z)-1)/(p(z)+1)$. The expressions of $a_i (i=2,3,4,5)$ are obtained in terms of $p_j (j=1,2,3,4)$ by substituting $w(z)$, $p(z)$, and $f(z)$ in equation (\ref{3 formulaai}) with suitable comparison of coefficients so that
\begin{equation}\label{6 alphaa2}
a_2=p_1(1-\alpha),
\end{equation}
\begin{equation}\label{6 alphaa3}
a_3=\dfrac{(1-\alpha)}{2}\bigg(p_2+p_1^2(1-\alpha)\bigg),
\end{equation}
\begin{equation}\label{6 alphaa4}
    a_4=\dfrac{(1-\alpha)}{6}\bigg(2p_3+3p_1p_2(1-\alpha)+p_1^3(1-\alpha)^2\bigg)
\end{equation}
and
\begin{equation}\label{6 alphaa5}
a_5=\dfrac{(1-\alpha)}{24}\bigg(6p_4+(1-\alpha)\bigg(3p_2^2+8p_1p_3\bigg)+(1-\alpha)^2\bigg(6p_1^2p_2+p_1^4(1-\alpha)\bigg)\bigg).
\end{equation}
	
The formula for $p_j$ $(j=2,3,4)$, which plays a significant role in finding the sharp bound of the Hankel determinant and has been prominently exploited in the main theorem, is contained in the Lemma \ref{pformula} below. %For further details on the class $\mathcal{P}$ coefficients, refer to the survey article, \cite{surveycaratheodory}.
\begin{lemma}\cite{rj,lemma1}\label{pformula}
Let $p\in \mathcal {P}$ has the form $1+\sum_{n=1}^{\infty}p_n z^n.$ Then
\begin{equation*}
2p_2=p_1^2+\gamma (4-p_1^2),
\end{equation*}
\begin{equation*}
4p_3=p_1^3+2p_1(4-p_1^2)\gamma -p_1(4-p_1^2) {\gamma}^2+2(4-p_1^2)(1-\vert \gamma\vert ^2)\eta,
\end{equation*}
and \begin{align*}
8p_4&=p_1^4+(4-p_1^2)\gamma (p_1^2({\gamma}^2-3\gamma+3)+4\gamma)-4(4-p_1^2)(1-\vert \gamma\vert ^2)(p_1(\gamma-1)\eta\\
&\quad+\bar{\gamma}{\eta}^2-(1-\vert \eta\vert ^2)\rho),
\end{align*}
for some $\gamma$, $\eta$ and $\rho$ such that $\vert \gamma\vert \leq 1$, $\vert \eta\vert \leq 1$ and $\vert \rho\vert \leq 1.$
\end{lemma}

\section{Sharp $H_{3}(1)$ for $\mathcal{S}^{*}(\alpha)$}	
\noindent Recently, Kowalczyk et al. \cite{4/9} and Banga and Kumar \cite{banga} obtained the  sharp bound of the third-order Hankel determinant for functions in the class $\mathcal{S}^{*}:=\mathcal{S}^{*}(0)$, independently whereas Rath et al. \cite{rath} determined the sharp bound of $H_3(1)$ for functions in the class $\mathcal{S}^{*}(1/2)$ and corrected the proof provided in \cite{lecko 1/2 bound}. In this section, we extend our analysis to calculate the sharp bound of $H_{3}(1)$ for functions in the class $\mathcal{S}^{*}(\alpha)$ for some additional range of $\alpha$. Below, is our main result.%All the graphical representations in this article have been generated using MATLAB.

\begin{theorem}\label{6 sharph31}
Let $f\in \mathcal {S}^{*}(\alpha)$. Then
\begin{equation}
\vert H_{3}(1)\vert \leq \frac{4(1-\alpha)^2}{9}, \quad \alpha\in[0,1/6]\cup \{1/2\}.\label{6 9.5}
\end{equation}
This result is sharp.
\end{theorem}

\begin{proof}
Since, the class $\mathcal{P}$ is invariant under rotation, we have $p_1\in [0,2]$ and assume $p_1=:p$. The expressions of $a_ i$ $(i=2,3,4,5)$ from equations \eqref{6 alphaa2}-\eqref{6 alphaa5} are substituted in equation (\ref{1h3}). We get
\begin{align*}
H_{3}(1)&=\dfrac{(1-\alpha)^2}{144}\bigg(-(1-\alpha)^4p^6 + 3(1-\alpha)^3 p^4 p_2 +8(1-\alpha)^2 p^3 p_3+24(1-\alpha)p p_2 p_3\\
&\quad \quad\quad \quad \quad-18(1-\alpha)p^2p_4-9(1-\alpha)p_2^3-9(1-\alpha)^2p^2p_2^2 -16 p_3^2+18 p_2 p_4\bigg).
\end{align*}
After simplifying the calculations through Lemma \ref{pformula}, we obtain
\begin{equation*}					
H_{3}(1)=\dfrac{1}{1152}\bigg(\Delta_1(p,\gamma)+\Delta_2(p,\gamma)\eta+\Delta_3(p,\gamma){\eta}^2+\phi(p,\gamma,\eta)\rho\bigg),\quad \text{for}\quad\gamma,\eta,\rho\in \mathbb {D}.
\end{equation*}
Here
\begin{align*}
\Delta_1(p,\gamma):&=(1-\alpha)^2\bigg(\alpha(1-2\alpha)^2(3-2\alpha)p^6-(2-15\alpha+18\alpha^2)p^2{\gamma}^2(4-p^2)^2\\
&\quad+p^2{\gamma}^4(4-p^2)^2-(10-15\alpha)p^2{\gamma}^3(4-p^2)^2+36\alpha{\gamma}^3(4-p^2)^2\\
&\quad+(3-12\alpha^3+32\alpha^2-19\alpha)p^4{\gamma}(4-p^2)-9(1-2\alpha)p^4{\gamma}^3(4-p^2)\\
&\quad+(3-16\alpha^2+2\alpha)p^4{\gamma}^2(4-p^2)-36(1-2\alpha)p^2{\gamma}^2(4-p^2)\bigg),\\
\Delta_2(p,\gamma):&=4(1-\vert \gamma\vert ^2)(4-p^2)(1-\alpha)^2\bigg((8\alpha^2-10\alpha+3)p^3+9(1-2\alpha)p^3{\gamma}\\
&\quad\quad\quad\quad\quad\quad \quad\quad\quad\quad \quad\quad\quad+(5-12\alpha)p\gamma (4-p^2)-p{\gamma}^2(4-p^2)\bigg),\\
\Delta_3(p,\gamma):&=4(1-\vert \gamma\vert ^2)(4-p^2)(1-\alpha)^2\bigg(-8(4-p^2)-\vert \gamma\vert ^2(4-p^2)+9(1-2\alpha)p^2\bar{\gamma}\bigg),\\
\phi(p,\gamma,\eta):&=36(1-\vert \gamma\vert ^2)(4-p^2)(1-\vert \eta\vert ^2)(1-\alpha)^2\bigg((4-p^2)\gamma-(1-2\alpha)p^2\bigg).
\end{align*}

Assume $x:=\vert \gamma\vert $, $y:=\vert \eta\vert $ and since $\vert \rho\vert \leq 1,$ the above expression reduces to
\begin{align*}
\vert H_{3}(1)\vert \leq \dfrac{1}{1152}\bigg(\vert \Delta_1(p,\gamma)\vert +\vert \Delta_2(p,\gamma)\vert y+\vert \Delta_3(p,\gamma)\vert y^2+\vert \phi(p,\gamma,\eta)\vert \bigg)\leq Z(p,x,y),
\end{align*}
where
\begin{equation}
Z(p,x,y)=\dfrac{1}{1152}\bigg(z_1(p,x)+z_2(p,x)y+z_3(p,x)y^2+z_4(p,x)(1-y^2)\bigg)\label{3 new}
\end{equation}
with
\begin{align*}
z_1(p,x):&=(1-\alpha)^2\bigg(\alpha(1-2\alpha)^2(3-2\alpha)p^6+(2-15\alpha+18\alpha^2)p^2x^2(4-p^2)^2\\
&\quad\quad\quad\quad\quad +p^2x^4(4-p^2)^2+(10-15\alpha)p^2x^3(4-p^2)^2+36\alpha x^3(4-p^2)^2\\
&\quad\quad\quad\quad\quad+(3-12\alpha^3+32\alpha^2-19\alpha)p^4x(4-p^2)+9(1-2\alpha)p^4x^3(4-p^2)\\
&\quad\quad\quad\quad\quad+(3-16\alpha^2+2\alpha)p^4x^2(4-p^2)+36(1-2\alpha)p^2x^2(4-p^2)\bigg),\\
z_2(p,x):&=4(1-x^2)(4-p^2)(1-\alpha)^2\bigg((8\alpha^2-10\alpha+3)p^3+9(1-2\alpha)p^3x\\
&\quad\quad\quad\quad\quad\quad\quad\quad\quad\quad\quad\quad\quad+(5-12\alpha)px(4-p^2)+px^2(4-p^2)\bigg),\\
z_3(p,x):&=4(1-x^2)(4-p^2)(1-\alpha)^2\bigg(8(4-p^2)+x^2(4-p^2)+9(1-2\alpha)p^2x\bigg),\\
z_4(p,x):&=36(1-x^2)(4-p^2)(1-\alpha)^2\bigg((4-p^2)x+(1-2\alpha)p^2\bigg).
\end{align*}

Note that for $\alpha\in[0,1/6]$, all the factors involving $\alpha$ in $\vert \Delta_1(p,\gamma)\vert $, $\vert \Delta_2(p,\gamma)\vert $, $\vert \Delta_3(p,\gamma)\vert $ and $\vert \phi(p,\gamma,\eta)\vert $, are positive as $1/6$ is the smallest positive root of the equation $2-15\alpha+18\alpha^2=0$. We maximise $Z(p,x,y)$ within the closed cuboid $Y: [0,2] \times [0,1] \times [0,1]$, by finding the maximum values in the interior of $Y$, in the interior of the six faces and on the twelve edges.\\

\noindent \underline{{\bf{Case I:}}}\\
We begin with every interior point of $Y$ assuming $(p,x,y)\in (0,2)\times (0,1)\times (0,1)$. We determine $\partial{Z}/\partial y$
 to examine the points of maxima in the interior of $Y$. Thus
\begin{align*}
\dfrac{\partial Z}{\partial y}&=\dfrac{(4 - p^2)(1 - x^2)(1-\alpha^2)}{288}  \bigg(8(8-9x+x^2)y- 2 p^2 (1-x)y(17-x-18\alpha)\\
&\quad \quad\quad \quad\quad\quad\quad\quad\quad\quad\quad\quad+p^3(3-x^2-10\alpha+8\alpha^2+x(4-6\alpha))\\
&\quad \quad\quad \quad\quad\quad\quad\quad\quad\quad\quad\quad+4xp(5+x-12\alpha)\bigg).
\end{align*}
Now, $\partial Z/\partial y=0$ gives
\begin{equation*}
y=y_0:=\dfrac{4xp(5+x-12\alpha)+p^3(3-x^2-10\alpha+8\alpha^2+x(4-6\alpha))}{2(1-x) (-4(8-x) + p^2 (17-x-18\alpha))}.
\end{equation*}
The existence of critical points require that $y_0\in(0,1)$ and can only exist when

\begin{align}
2p^2(1-x)(17-x-18\alpha)&>-p^3(-3+x^2+10\alpha-8\alpha^2-x(4-6\alpha))\nonumber\\
&\quad+4px(5+x-12\alpha)+8(1-x)(8-x).\label{6 h1}
\end{align}
We try finding the solution satisfying the inequality (\ref{6 h1}) for the critical points. The possible range for which $y_0\in (0,1)$, is $(0,(3+2\alpha)/9)\times (\tilde{p},2)$. Here
\begin{align*}
    \tilde{p}:&= \frac{2(17-18\alpha)}{\tilde{p_1}}-\frac{ 2^{4/3}(1-i\sqrt{3})(17-18\alpha)^2}{\tilde{p_1}(\tilde{p_2}+\sqrt{-256(17-18\alpha)^6+(\tilde{p_2})^2})^{1/3}}\\
    &\quad-\frac{(1+i\sqrt{3}) (\tilde{p_2}+\sqrt{-256(17-18\alpha)^6+(\tilde{p_2})^2})^{1/3}}{2^{4/3}\tilde{p_1}}
\end{align*}
where
\begin{align*}
    \tilde{p_1}:&=3(3-10\alpha+8\alpha^2);\\
    \tilde{p_2}:&=63056-146016\alpha+8640\alpha^2+183168\alpha^3-110592\alpha^4.
\end{align*}
Therefore, a calculation reveals that the maxima attained in the interior of $Y$ at each such $y_0\in (0,1)$ is always less than $4(1-\alpha)^2/9$ for $\alpha\in[0,1/6]$.
					
\noindent  \underline{{\bf{Case II:}}} \\
The interior of six faces of the cuboid $Y$, is now under consideration, for the further calculations.\\
\underline{On $p=0$}, $Z(p,x,y)$ turns into
\begin{equation}
s_1(x,y):=\dfrac{(1-\alpha)^2( (1 - x^2) ((8 + x^2) y^2 + 9 x (1 - y^2) )+ 9 x^3\alpha)}{18},\quad x,y\in (0,1).\label{6 9.4}
\end{equation}
Since
\begin{equation*}
\dfrac{\partial s_1}{\partial y}=\dfrac{(1 - x^2)(x+1)(8-x)(1-\alpha)^2y}{9}\neq 0,\quad x,y\in (0,1).
\end{equation*}
Thus, $s_1$ has no critical point in $(0,1)\times(0,1)$.
						
\noindent \underline{On $p=2$}, $Z(p,x,y)$ reduces to
\begin{equation}
Z(2,x,y):=\dfrac{\alpha (1-\alpha)^2 (1-2\alpha)^2(3-2\alpha)}{18},\quad x,y\in (0,1).\label{6 9.3}
\end{equation}

\noindent \underline{On $x=0$}, $Z(p,x,y)$ becomes
\begin{align}
s_2(p,y):&=\dfrac{(1-\alpha)^2}{1152}\bigg(\alpha (3-2\alpha)(1-2\alpha)^2p^6 + 36(1-2\alpha)p^2(4-p^2)(1-y^2)\nonumber\\
&\quad\quad\quad\quad\quad\quad+32 (4- p^2)^2 y^2+4(3-10\alpha+8\alpha^2)p^3y(4-p^2)\bigg) \label{6 9.1}
\end{align}
with $p\in (0,2)$ and $y\in (0,1)$. On solving $\partial s_2/\partial p$ and $\partial s_2/\partial y$, to find the points of maxima. After resolving $\partial s_2/\partial y=0,$ we get
\begin{equation}
y=\dfrac{p^3(3-10\alpha+8\alpha^2)}{2(17p^2-32-18p^2\alpha)}(=:y_0).\label{6 y}
\end{equation}
Upon calculations, we observe that to have $y_0\in (0,1)$ for the given range of $y$, $p=:p_0>\approx A(\alpha)$ (see Fig. \ref{picture2}) is needed with $\alpha\in[0,\beta_0)$. This $\beta_0\in [0,1)$ is the smallest positive root of $-3+10\alpha-8\alpha^2=0$ and no such $p\in (0,2)$ exists when $\alpha\in (\beta_0,1)$. It is to be noted that the expression of $A(\alpha)$ is complex but the coefficient of its imaginary part for $\alpha\in[0,1/6]$ is highly negative (of order $10^{-15})$, which can be neglected and $A(\alpha)$ can be treated as a real number.
Here,
\begin{align*}
    A(\alpha)&:=\frac{1}{3(-3+10\alpha-8\alpha^2)}\bigg(2(18\alpha-17)-\frac{2^{4/3}(1-i\sqrt{3})(18\alpha-17)^2}{B}-\frac{(1+i\sqrt{3})B}{2^{4/3}}\bigg),
\end{align*}
with
\begin{equation*}
 B:=\bigg(C+\sqrt{-256(18\alpha-17)^6+C^2}\bigg)^{1/3}
 \end{equation*}
 and
 \begin{equation*}
  C:=-63056+146016\alpha-8640\alpha^2-183168\alpha^3+110592\alpha^4.
\end{equation*}
Based on computations, $\partial s_2/\partial p=0$ gives
\begin{align}
0&=16p(9-18\alpha-y^2(25-18\alpha))-2p^2y(3-10\alpha+8\alpha^2)(5p^2-12)\nonumber\\
&\quad +3\alpha (1-2\alpha)^2 (3-2\alpha)p^5-8p^3(9-18\alpha-y^2(17-18\alpha)).\label{6 9}
\end{align}
After substituting equation (\ref{6 y}) into equation (\ref{6 9}), we have
\begin{align}
0&=p\bigg(49152  (1-2\alpha)-3072 p^2 (25 -68\alpha+36\alpha^2)-p^8(1-2\alpha)^2(153-1437\alpha\nonumber\\
&\quad+3118\alpha^2-2484\alpha^3+648\alpha^4)+16p^4(2427-7890\alpha+6020\alpha^2+616\alpha^3-1024\alpha^4)\nonumber\\
&\quad\quad-128p^6(48-153\alpha+20\alpha^2+340\alpha^3-352\alpha^4+96\alpha^5)\bigg).\label{6 40}
\end{align}
A numerical calculation suggests that the solution of (\ref{6 40}) in the interval $(0,2)$ is $p\approx B(\alpha)$ whenever $\alpha\in[0,\alpha_2)$, where $\alpha_2\in[0,1)$ is the smallest positive root of $153-1437\alpha+3118\alpha^2-2484\alpha^3+648\alpha^4=0$, otherwise no such $p\in (0,2)$ exists, see Fig. \ref{picture2}. Thus, $s_2$ does not have any critical point in $(0,2)\times(0,1)$.\\
Here
\begin{align*}
    B(\alpha):&=\frac{1}{\sqrt{2}}\bigg[\bigg\{ \frac{1}{F}\bigg( -3072+3648\alpha+6016\alpha^2-9728\alpha^3+3072\alpha^4\\
    &\quad\quad\quad-\frac{4F}{\sqrt{3}}\bigg\{\frac{1}{E^2}\bigg(\frac{768J^2}{(1-2\alpha)^2}+\frac{2GE}{(1-2\alpha)}+\frac{HE}{I}+\frac{IE}{(1-2\alpha)^2}\bigg)\bigg\}^{1/2}\\
    &\quad\quad\quad +\frac{4F}{\sqrt{3}}\bigg\{\frac{-1}{E^3}\bigg(\frac{-1536J^2E}{(1-2\alpha)^2} -\frac{4GE^2}{(1-2\alpha)}+\frac{HE^2}{I}+\frac{IE^2}{(1-2\alpha)^2} \\
    &\quad\quad\quad -\frac{K}{(1-2\alpha)^3\bigg\{\frac{1}{E^2}\bigg( \frac{768J^2}{(1-2\alpha)^2}+\frac{2GE}{(1-2\alpha)}+\frac{HE}{I}+\frac{IE}{(1-2\alpha)^2}\bigg)\bigg\}^{1/2}} \bigg)
    \bigg\}^{1/2}\bigg) \bigg\}^{1/2}  \bigg],
    \end{align*}
    with
    \begin{align*}
    E&:=153-1437\alpha+3118\alpha^2-2484\alpha^3+648\alpha^4;\\
    F&:=(1-2\alpha)E;\\
    G&:=2427-3036\alpha-52\alpha^2+512\alpha^3;\\
    H&:=8217-173160\alpha+1260312\alpha^2-2415264\alpha^3+2091664\alpha^4-1048576\alpha^5+262144\alpha^6;
    \end{align*}
   \begin{align*}
    I&:=\bigg(127065213-1886889978\alpha+12579196752\alpha^2-49871499552\alpha^3+132494582880\alpha^4\\
    &\quad\quad -253944918720\alpha^5+368411062528\alpha^6-410152327680\alpha^7+340236674304\alpha^8\\
    &\quad\quad -196757891584\alpha^9+71861010432\alpha^{10}-14168358912\alpha^{11}+1073741824\alpha^{12}\\
    &\quad\quad +288\sqrt{6}\bigg((1-2\alpha)^8(3-4\alpha)^2(3604621581 - 39763739565\alpha+202125486510\alpha^2\\
    &\quad\quad -633657349224\alpha^3+1436021769744 \alpha^4-2516421142080 \alpha^5+34829931648\alpha^6\\
     &\quad\quad-3872882513280\alpha^7+3466438619648\alpha^8-2402201403136\alpha^9+1198713174528\alpha^{10}\\
      &\quad\quad-394099818496\alpha^{11}+75581358080\alpha^{12}-6442450944\alpha^{13})\bigg)^{1/2}\bigg)^{1/3};\\
      J&:=-48+57\alpha+94\alpha^2-152\alpha^3+48\alpha^4;
      \end{align*}
      and
      \begin{align*}
      K&:=288\sqrt{3}\bigg(15963705 - 129546873\alpha+510658314\alpha^2-1308834456 \alpha^3+2415583204 \alpha^4\\
      &\quad\quad\quad\quad\quad -3321041560 \alpha^5+3420107120 \alpha^6-2619528992\alpha^7+1464766656\alpha^8\\
      &\quad\quad\quad\quad\quad -575732096\alpha^9+147709696\alpha^{10}-21284352 \alpha^{11}+1179648\alpha^{12}\bigg).
\end{align*}

\begin{figure}
    \centering
   \includegraphics[width=\textwidth] {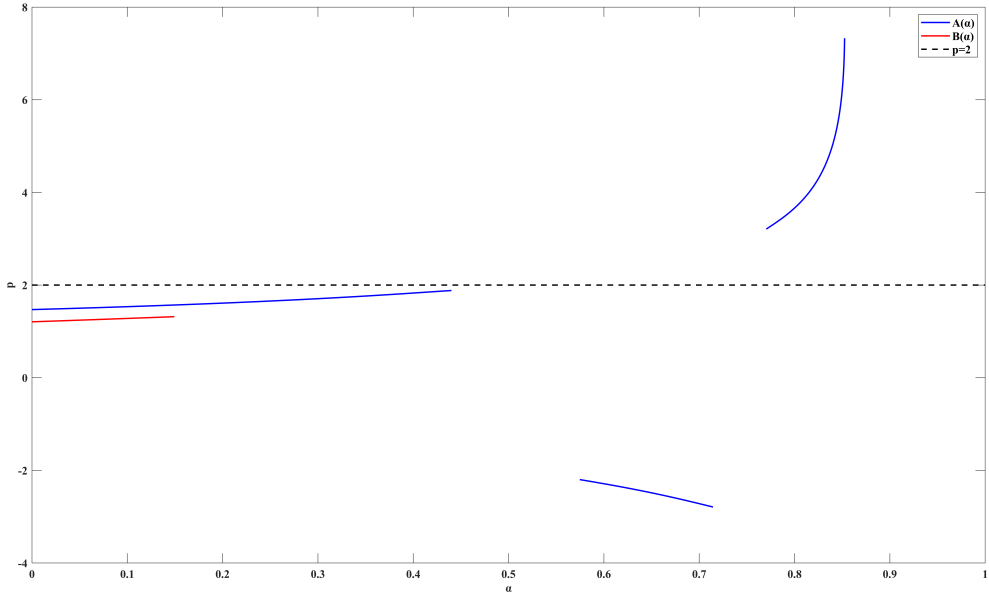}
    \caption{Graphical representation of $p$ versus $\alpha$. Here, $B(\alpha)$ (Red) and $A(\alpha)$ (blue) do not intersect for any choice of $\alpha$. Dashed black line represents $p=2$.}
    \label{picture2}
\end{figure}

\noindent \underline{On $x=1$}, $Z(p,x,y)$ reduces into
\begin{align}
s_3(p,y):&=\dfrac{(1-\alpha)^2}{576}\bigg(288\alpha+16p^2(11-33\alpha+9\alpha^2)-8p^4(5-13\alpha+5\alpha^2+3\alpha^3)\nonumber\\
    &\quad\quad\quad\quad\quad\quad-p^6(1-4\alpha+6\alpha^2-16\alpha^3+4\alpha^4)\bigg), \quad p\in (0,2).\label{6 9.2}
\end{align}
While computing $\partial s_3/\partial p=0$, $p=:p_0\approx 2L(\alpha)$ for $\alpha\in [0,\alpha_0)\cup(\alpha_0,\alpha_1)$, comes out to be the critical point, where $\alpha_0\in[0,1)$ is the smallest positive root of $1-4\alpha+6\alpha^2-16\alpha^3+4\alpha^4=0$ and $\alpha_1(\approx 0.370803927)\in[0,1)$ (see Fig. \ref{picture1}) is the largest value so that $p\in (0,2)$ otherwise no such real $p\in (0,2)$ exists beyond this $\alpha_1$.
Here
\begin{equation}
\begin{aligned}
 \left.
\begin{array}{cc}
     & L(\alpha):=\sqrt{\dfrac{-10+26\alpha-10\alpha^2-6\alpha^3+M}{3N}};\\
%l:&=-10+26\alpha-10\alpha^2-6\alpha^3+M;\\
& M:=\sqrt{133-751\alpha+1497\alpha^2-1630\alpha^3+1666\alpha^4-708\alpha^5+144\alpha^6};\\
& N:=1-4\alpha+6\alpha^2-16\alpha^3+4\alpha^4.\label{6 supportpalpha}
\end{array}
\right\}
\end{aligned}
\end{equation}
Undergoing simple calculations, $s_3$ achieves its maximum value, approximately equals $P(\alpha)$, %$\alpha \in [0,1/6]$ at $p_0$. %
$[0,\alpha_0)\cup(\alpha_0,\alpha_1)$ at $p_0$.
Here
\begin{equation}
\begin{aligned}
 \left.
\begin{array}{cc}
   & P(\alpha):=\dfrac{(1-\alpha)^2}{486}\bigg(243\alpha-\frac{18(11-33\alpha+9\alpha^2)(10-26\alpha+10\alpha^2+6\alpha^3-M)}{N}\\
    &\quad\quad\quad\quad\quad\quad\quad\quad-\frac{12(5-13\alpha+5\alpha^2+3\alpha^3)(-10+26\alpha-10\alpha^2-6\alpha^3+M)^2}{N^2}\\
    &\quad\quad\quad\quad-\frac{2(-10+26\alpha-10\alpha^2-6\alpha^3+M)^3}{N^2} \bigg).
    \end{array}
\right\}
\end{aligned}\label{6 palpha}
\end{equation}

\begin{figure}
    \centering
   \includegraphics[width=\textwidth]{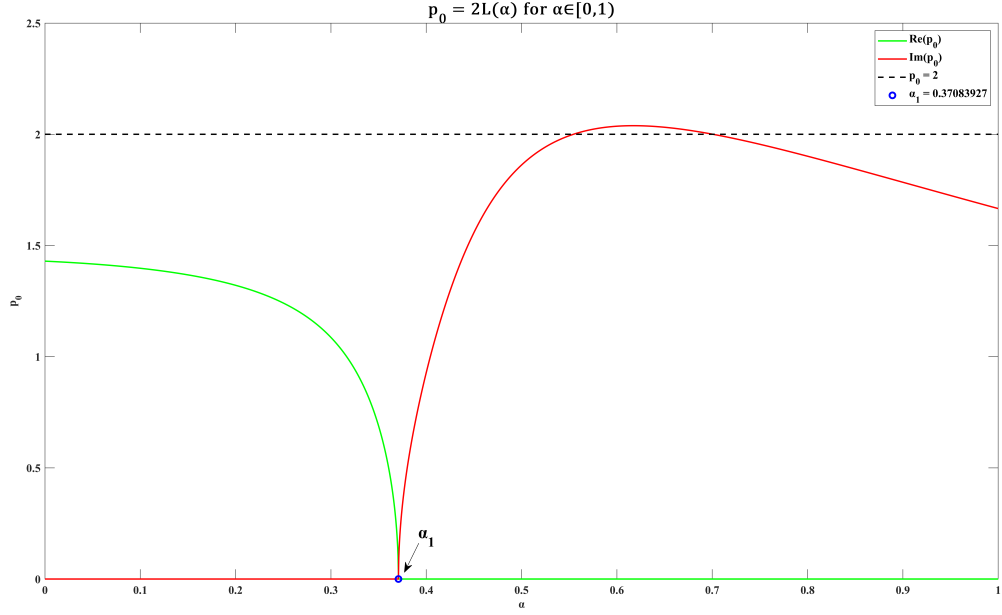}
    \caption{Graphical representation of $p_0$ versus $\alpha$. Here, $\RE(p_0)$ (green) and $\IM(p_0)$ (red) represent the value of $p_0$ at different $\alpha$, where $\alpha_1$ (blue circle) is the point at which $p_0$ transforms from completely real to imaginary. Dashed black line represents $p_0=2$. }
    \label{picture1}
\end{figure}

\noindent \underline{On $y=0$}, $Z(p,x,y)$ can be seen as
\begin{align*}
s_4(p,x):&=\dfrac{(1-\alpha)^2}{1152}\bigg(576\bigg(x-x^3(1-\alpha)\bigg)+16p^2\bigg(9-18x+x^4+x^3(28-33\alpha)\\
&\quad \quad\quad\quad\quad-18\alpha+x^2(2-15\alpha+18\alpha^2)\bigg) -4p^4\bigg(9+2x^4+x^3(20-21\alpha) \\
&\quad \quad\quad\quad\quad-18\alpha+x^2(1-32\alpha+52\alpha^2)+x(-12+19\alpha-32\alpha^2+12\alpha^3)\bigg)\\
&\quad\quad\quad\quad\quad+p^6\bigg (x^4+\alpha(1-2\alpha)^2(3-2\alpha)+x^3(1+3\alpha)\\
&\quad \quad\quad\quad\quad-x^2(1+17\alpha-34\alpha^2)+x(-3+19\alpha-32\alpha^2+12\alpha^3)\bigg)\bigg).
\end{align*}

Furthermore, through some calculations, such as
\begin{align*}
\dfrac{\partial s_4}{\partial x}&=\dfrac{(1-\alpha)^2}{1152}\bigg(576\bigg(1-3x^2(1-\alpha)\bigg)-16p^2\bigg(18-4x^3-3x^2(28-33\alpha)\\
&\quad\quad \quad\quad \quad \quad-2x(2-15\alpha+18\alpha^2)\bigg)+p^6\bigg(4x^3+3x^2(1+3\alpha)-2x(1+17\alpha\\
&\quad\quad \quad\quad \quad \quad-34\alpha^2)-3+19\alpha-32\alpha^2+12\alpha^3\bigg)-4p^4\bigg(8x^3+3x^2(20-21\alpha)\\
&\quad\quad \quad\quad \quad \quad+2x(1-32\alpha+52\alpha^2)-12+19\alpha-32\alpha^2+12\alpha^3\bigg)\bigg)
\end{align*}

and \begin{align*}
\dfrac{\partial s_4}{\partial p}&=\dfrac{(1-\alpha)^2}{1152}\bigg(32 p\bigg (9 - 18 x + x^4 + x^3 (28 - 33\alpha)-18\alpha+x^2(2-15\alpha+18\alpha^2)\bigg) \\
&\quad\quad \quad\quad\quad\quad -16 p^3 \bigg(9  + 2 x^4 + x^3 (20 - 21 \alpha)-18\alpha+x^2(1-32\alpha+52\alpha^2)\\
&\quad\quad \quad\quad\quad\quad+x(-12+19\alpha-32\alpha^2+12\alpha^3)\bigg)+6p^5\bigg(x^4+\alpha(1-2\alpha)^2(3-2\alpha)\\
&\quad\quad \quad\quad\quad\quad+x^3(1+3\alpha)-x^2(1+17\alpha-34\alpha^2)+x(-3+19\alpha-32\alpha^2+12\alpha^3)\bigg)\bigg),
\end{align*}
indicates that there does not exist any common solution for the system of equations $\partial s_4/\partial x=0$ and $\partial s_4/\partial p=0$, thus, $s_4$ has no critical points in $(0,2)\times(0,1)$.\\

\noindent \underline{On $y=1$}, $Z(p,x,y)$ reduces to
\begin{align*}
s_5(p,x):&=\dfrac{(1-\alpha)^2}{1152}\bigg(64px(1-x^2)(5+x-12\alpha)+64(8-7x^2-x^4+9\alpha)\\
&\quad\quad \quad\quad\quad+16p^3(1-x^2)(3-x-2x^2-10\alpha+6x\alpha+8\alpha^2)-2p^6(1-4\alpha\\
&\quad\quad \quad\quad\quad-16\alpha^3+4\alpha^4)+16p^2\bigg(6+14x^2+2x^4+x(9-18\alpha)-66\alpha\\
&\quad\quad \quad\quad\quad +18\alpha^2-9x^3(1-2\alpha)\bigg)+4p^5(1-x^2)\bigg(x^2-3+10\alpha-8\alpha^2\\
&\quad\quad \quad\quad\quad-x(4-6\alpha)\bigg)-4p^4\bigg(7x^2+x^4+x(9-18\alpha)-9x^3(1-2\alpha)\\
&\quad\quad \quad\quad\quad+4(3-13\alpha+5\alpha^2+3\alpha^3)\bigg)\bigg).
\end{align*}
We note that the equations $\partial s_5/\partial x=0$ and $\partial s_5/\partial p=0$ possess no common solution in $(0,2)\times (0,1).$\\

\noindent \underline{{\bf{Case III:}}}
Now, we determine the maximum values that $Z(p,x,y)$ may obtain on the edges of the cuboid $Y$.\\
From equation (\ref{6 9.1}), we have
\begin{equation*}
Z(p,0,0)=r_1(p):=\frac{p^2(1-\alpha)^2(1-2\alpha)(144-36p^2+p^4\alpha(3-8\alpha+4\alpha^2)}{1152}.
\end{equation*}
Here, we consider the following three subcases for different choices of $\alpha$.
\begin{enumerate}
 \item For $\alpha=0$, $r_1(p)$ reduces to $p^2(4-p^2)/32$ and $r_1'(p)=0$ for $p=0$, the point of minima and $p=\sqrt{2}$, the point of maxima. Therefore
    \begin{equation*}
Z(p,0,0)\leq \frac{1}{8}, \quad p\in [0,2].
\end{equation*}
\item For $\alpha=1/2$, $r_1(p)=0$.\\
 \item For $\alpha=(0,1/2)\cup (1/2,1),$
$r'_1(p)=p(1-\alpha)^2(1-2\alpha)(48-24p^2+p^4\alpha(3-8\alpha+4\alpha^2))=0$ for $p=0$ and $p=2\bigg(((3-R(\alpha))/(3\alpha-8\alpha^2+4\alpha^3)\bigg)^{1/2}$ as the points of minima and maxima respectively. So,
\begin{equation*}
    Z(p,0,0)\leq R_0(\alpha):=\frac{(1-\alpha)^2(3-R(\alpha))(-3+6\alpha-16\alpha^2+8\alpha^3+R(\alpha))}{6(3-2\alpha)^2\alpha^2(1-2\alpha)},
\end{equation*}
with $R(\alpha):=\sqrt{3(3-3\alpha+8\alpha^2-4\alpha^3)}$.
\end{enumerate}
Now, equation (\ref{6 9.1}) at $y=1,$ implies that $Z(p,0,1)=r_2(p):=(1-\alpha)^2(32(4-p^2)^2+\alpha (1-2\alpha)^2(3-2\alpha)p^6+4p^3(4-p^2)(3-10\alpha+8\alpha^2))/1152.$ Note that $r_2'(p)$ is a decreasing function in $[0,2]$ and hence $p=0$ becomes the point of maxima. Thus
\begin{equation*}
Z(p,0,1)\leq \dfrac{4(1-\alpha)^2}{9}, \quad p\in [0,2].
\end{equation*}
Through calculations, equation (\ref{6 9.1}) shows that $Z(0,0,y)$ attains its maximum value at $y=1,$ which implies that
\begin{equation*}
Z(0,0,y)\leq \dfrac{4(1-\alpha)^2}{9}, \quad y\in [0,1].
\end{equation*}
Since, the equation (\ref{6 9.2}) is free from $y$, we have \begin{align*}
Z(p,1,1)=Z(p,1,0)=r_3(p):&=\frac{(1-\alpha)^2}{576}\bigg(288\alpha+16p^2(11-33\alpha+9\alpha^2)\\
&\quad\quad\quad\quad\quad\quad-8p^4(5-13\alpha+5\alpha^2+3\alpha^3)\\
&\quad\quad\quad\quad\quad\quad-p^6(1-4\alpha+6\alpha^2-16\alpha^3+4\alpha^4)\bigg).
\end{align*}
Now, $r_3'(p)=32p(11-33\alpha+9\alpha^2)-32p^3(5-13\alpha+5\alpha^2+3\alpha^3)-6p^5(1-4\alpha+6\alpha^2-16\alpha^3+4\alpha^4)=0$ when $p=\delta_1:=0$ and $p=\delta_2:=2L(\alpha)$ for $\alpha\in[0,\alpha_0)\cup(\alpha_0,\alpha_1)$,
as the points of minima and maxima respectively, in the interval $[0,2]$. The justification of $P(\alpha)$, $\alpha_0$ and $\alpha_1$ are provided above through equation \eqref{6 supportpalpha} and \eqref{6 palpha}. Thus, from equation (\ref{6 9.2}),
\begin{equation*}
Z(p,1,1)=Z(p,1,0)\leq P(\alpha),\quad p\in [0,2]\quad\text{and}\quad \alpha\in [0,\alpha_0)\cup(\alpha_0,\alpha_1).
\end{equation*}
Consider equation (\ref{6 9.2}) at $p=0$, we get
\begin{equation*}
Z(0,1,y)=\frac{\alpha (1-\alpha)^2}{2}.
\end{equation*}

Equation (\ref{6 9.3}) indicates that
\begin{equation*}
Z(2,1,y)=Z(2,0,y)=Z(2,x,0)=Z(2,x,1)=\dfrac{\alpha (1-2\alpha)^2(1-\alpha)^2(3-2\alpha)}{18}.
\end{equation*}
Using equation (\ref{6 9.4}), $Z(0,x,1)=r_4(x):=(1-\alpha)^2 (8-7x^2-x^4+9x^3\alpha)/18.$ Upon calculations, we see that $r_4$ is a decreasing function of $x$ in $[0,1]$ and therefore $x=0$ is the point of maxima. Hence
\begin{equation*}
Z(0,x,1)\leq \dfrac{4(1-\alpha)^2}{9},\quad x\in [0,1].
\end{equation*}
On again using equation (\ref{6 9.4}), $Z(0,x,0)=r_5(x):=x(1-(1-\alpha)x^2)(1-\alpha)^2/2.$ Moreover, $r_5'(x)=0$ when $x=\delta_3:=1/\sqrt{3(1-\alpha)}.$ Observe that $r_5(x)$ increases in $[0,\delta_3)$ and decreases in $(\delta_3,1].$ Hence,
\begin{equation*}
Z(0,x,0)\leq \frac{(1-\alpha)^2}{3\sqrt{3(1-\alpha)}},\quad x\in [0,1].
\end{equation*}

We also provide a graphical representation of six upper-bounds (u.b) of $H_{3}(1)$ in Fig. \ref{picture8}. Given all the cases, the sharp inequality $\vert H_{3}(1)\vert \leq 4(1-\alpha)^2/9$, holds for every $\alpha\in[0,1/6]\cup\{1/2\}$.

\begin{figure}%[htp]
    \centering
   \includegraphics[width=\textwidth]{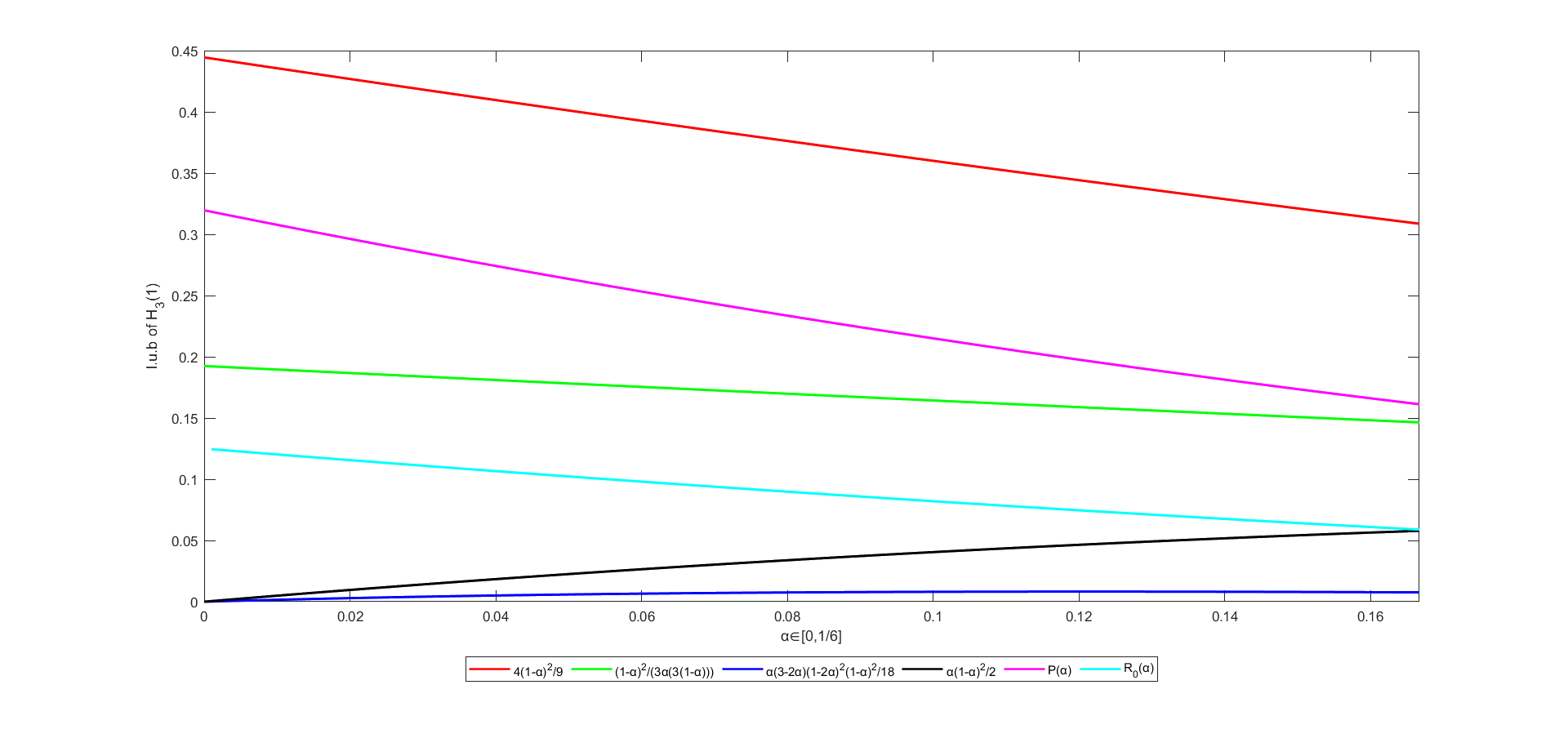}
    \caption{Graph of six upper-bounds (u.b) versus $\alpha$. The upper-bounds (u.b) of $H_{3}(1)$ are $4(1-\alpha)^2/9$ (red), $(1-\alpha)^2/(3\sqrt{3(1-\alpha)})$ (green), $\alpha (1-2\alpha)^2(1-\alpha)^2(3-2\alpha)/18$ (blue), $\alpha(1-2\alpha)^2/2 $ (black), $P(\alpha)$ (pink) and $R_{0}(\alpha)$ (cyan) for $\alpha\in[0,1/6]$. }
    \label{picture8}
\end{figure}
%Given all the cases (see Fig. \ref{picture8}), the sharp inequality $\vert H_{3}(1)\vert \leq 4(1-\alpha)^2/9$, holds for every $\alpha\in[0,1/6]$.\\
Let the function $f_0\in \mathcal{S}^{*}(\alpha):\mathbb{D}\rightarrow \mathbb{C}$, be defined as
\begin{equation*}
f_0(z)=z\exp\bigg(\int_{0}^{z}\dfrac{\frac{1+(1-2\alpha)t^3}{1-t^3}-1}{t}dt\bigg)=z+\dfrac{2(1-\alpha)}{3}z^4+\dfrac{(1-\alpha)(5-2\alpha)}{9}z^7+\cdots,\label{6 extremal}
\end{equation*}
with $f_0(0)=0$ and $f_0'(0)=1$, plays the role of an extremal function for the inequality presented in equation \eqref{6 9.5} with $a_2=a_3=a_5=0$ and $a_4=2(1-\alpha)/3$. 			
\end{proof}

Now, we provide remarks which incorporate the bound of $\vert H_{3}(1)\vert $ for the class $\mathcal{S}^{*}$ and  $\mathcal{S}^{*}(1/2)$, which are subclasses of $\mathcal{S}^{*}(\alpha)$, given as follows:
\begin{remark}
   On substituting $\alpha=0$ in Theorem \ref{6 sharph31},  $\mathcal{S}^{*}(0)=\mathcal{S}^{*}$ and from equation (\ref{6 9.5}), we get $\vert H_{3}(1)\vert \leq 4/9$. This bound is sharp and coincides with that of Kowalczyk et al.\cite{4/9} and Banga and Kumar \cite{banga}.
\end{remark}

\begin{remark}
   On substituting $\alpha=1/2$ in Theorem \ref{6 sharph31},  $\mathcal{S}^{*}(1/2)$ and from equation (\ref{6 9.5}), we get $\vert H_{3}(1)\vert \leq 1/9$. This bound is sharp and coincides with that of Rath et al.\cite{rath}.
\end{remark}

For some already known sharp bounds of $H_{3}(1)$, regarding various choices of $\varphi(z)$, See Table \ref{table1}. We note that the same bound is not available for $\varphi(z):=1+z-z^3/3$. Hence, as an application of Theorem \ref{6 sharph31}, we provide a better bound of $\vert H_{3}(1)\vert $ for functions belonging to the class, $\mathcal{S}_{Ne}^{*}:=\mathcal{S}^{*}(1+z-z^3/3)$.

\begin{corollary}
If $f\in \mathcal{S}_{Ne}^{*}$. Then $\vert H_{3}(1)\vert \leq 32/81\approx 0.395062$.
\end{corollary}
\begin{proof}
From \cite{wani}, we have
\begin{equation*}
\min_{\vert z\vert =r}\RE(\varphi(z))=\begin{cases}1-r+\frac{1}{3}r^3,& r\leq 1/\sqrt{3}\\
1-\frac{1}{3}(1+r^2)^{3/2},& r\geq 1/\sqrt{3}.
\end{cases}
\end{equation*}
We note that $\alpha=\min_{\vert z\vert =r}\RE(\varphi(z))=1-2\sqrt{2}/3$ as $r$ tends to $1$. Now, substitution of $\alpha=1-2\sqrt{2}/3\approx 0.057191\in [0,1/6]$ in equation \eqref{6 sharph31} implies that $\vert H_{3}(1)\vert \leq 32/81\approx 0.395062$.
\end{proof}				
				
\noindent \underline{{\bf{Open Problem:}}}\\
We have attempted to provide the sharp bound of $H_{3}(1)$ for functions, $f\in\mathcal{S}^{*}(\alpha)$ for $\alpha\in[0,1/6]\cup\{1/2\}$ in Theorem \ref{6 sharph31}. Further, this result is still open for the remaining range of $\alpha$ in $[0,1)$.

%\subsection*{Acknowledgment}
%Neha Verma is thankful for the Research Grant from Delhi Technological University, New Delhi-110042.

\end{document}